\documentclass[12pt,leqno]{amsart}
\usepackage{amsmath, amssymb, amsthm, xcolor, verbatim, datetime}
\usepackage{hyperref}
\usepackage{mathrsfs}
\usepackage{amsbsy}
\usepackage{amsfonts}
\usepackage{latexsym}

\newtheorem{prevtheorem}{Theorem}
\newtheorem*{theorem*}{Theorem}

\newtheorem{theorem}{Theorem}[section]
\newtheorem{lemma}[theorem]{Lemma}

\newtheorem{proposition}[theorem]{Proposition}

\theoremstyle{definition}

\newtheorem{remark}[theorem]{Remark}

\newtheorem*{conjecture}{Conjecture}

\usepackage[margin=1.2in]{geometry}
\numberwithin{equation}{section}

\title[The Amit--Ashurst conjecture for  finite metacyclic  $p$-groups]{The Amit--Ashurst conjecture for finite\\  metacyclic  $p$-groups}

 \author[R.\,D.~Camina]{Rachel D. Camina
 }
 
 \address{Rachel D. Camina: 
 Fitzwilliam College, CB3 0DG Cambridge,  UK}
 \email{
 rdc26@dpmms.cam.ac.uk}

 \author[W. L.\ Cocke]{William L.\ Cocke
 }
\address{William L.\ Cocke:
Augusta University, Augusta, GA 30901, USA}
\email{
wcocke@augusta.edu}

 \author[A. Thillaisundaram]{Anitha Thillaisundaram
 }
 
\address{Anitha Thillaisundaram: 
Centre for Mathematical Sciences, Lund University, 223 62 Lund, Sweden}
 \email{
 anitha.thillaisundaram@math.lu.se}

 \keywords{Words, fibres of word maps, Amit--Ashurst conjecture, metacyclic $p$-groups}
 \subjclass[2010]{Primary  20F10;  Secondary 20D15}

\begin{document}
\maketitle

\begin{abstract}
The Amit conjecture about word maps on finite nilpotent groups has been shown to hold for certain classes of groups. The generalised Amit conjecture says that the probability of an element occurring in the image of a word map on a finite nilpotent group~$G$ is either~0, or at least~$1/|G|$. Noting the work of Ashurst, we name the generalised Amit conjecture the Amit--Ashurst conjecture and show that the Amit--Ashurst conjecture holds for  finite $p$-groups with a cyclic maximal subgroup.

\end{abstract}

\section{Introduction}

Recall that a word $w$ in $k$ variables is an element of some free group $\textbf{F}_k=
\langle x_1,\dots,x_k\rangle$. The word $w$ induces a word map on a finite group~$G$, which by abuse of notation we write as 
\[
w:G^k\rightarrow G,
\] 
where $G^k$ denotes the direct product of $k$ copies of~$G$. This map induces a probability distribution $P_{w,G}$ on $G$, where for $g\in G$,
\[
P_{w,G}(g) = \frac{\big|\{(g_1,\dots,g_k)\in G^k \mid w(g_1,\ldots,g_k)=g \}\big|}{|G|^k}.
\]
We also write
\[
N_{w,G}(g)=\big|\{(g_1,\dots,g_k)\in G^k \mid w(g_1,\ldots,g_k)=g\}\big|,
\]
or simply $N_w(g)$ when the group is clear.

 We denote by $G_w$ the set of word values of $w$ in~$G$, i.e. the set of elements $g\in G$ such that the equation $w=g$ has a solution in $G^k$. Additionally, the \textit{fibre} of $g$ in $G^k$ is 
\[
\{(g_1,\ldots, g_k)\in G^k\mid w(g_1,\ldots, g_k)=g\}.
\]

Nikolov and Segal~\cite[Thm.~1.3]{NikolovSegal}
gave a characterisation of finite nilpotent groups (and of finite solvable
groups) based on the probability function~$P_{w,G}$: 
A finite group~$G$ is nilpotent if and only if the values $P_{w,G}(g)$ are bounded away from zero as $g$ ranges over~$G_w$ and $w$ ranges over all group words. The bound is quite small; of the order $p^{-|G|}$ where $p$ is the largest prime divisor of~$|G|$.

This bound has been improved in some cases. For example, it is easy to show that for an abelian group $G$, a word $w$ and $g$  in $G_w$, then $P_{w,G}(g)\ge |G|^{-1}$, since any word map is a homomorphism over an abelian group. Additionally,    Ashurst \cite[Ch.~4]{ashurst} showed that the nilpotent dihedral groups and the generalised quaternion groups also satisfy this property. In this direction, Ashurst conjectured the following in her thesis \cite[Conj.~6.2.1]{ashurst}, which we call the Amit--Ashurst conjecture:

\begin{conjecture}[Amit--Ashurst]
Let $G$ be a finite nilpotent group. Then 
\[
P_{w,G}(g) \ge |G|^{-1},
\]
where $w$ ranges over all words and $g\in G_w$.
\end{conjecture}

The Amit--Ashurst conjecture is a generalisation of the Amit conjecture which asks whether for a finite nilpotent group $G$, the probability distribution induced by the word $w$ satisfies $P_{w,G}(1)\geq |G|^{-1}$ for all $w$.  We note that for all groups $G$, and all $g \in G$, the word $w = x$ satisfies $P_{w,G}(g)=|G|^{-1}.$ Hence, the bound proposed in the Amit--Ashurst conjecture is strict. The Amit--Ashurst conjecture first appeared as the generalised Amit conjecture in~\cite{CIT}. However, since Ashurst first asked the question, we believe the term Amit--Ashurst conjecture is more appropriate. 

Not much progress has been made on the Amit conjecture, or the Amit--Ashurst conjecture. Since a finite nilpotent group is a direct product of its Sylow
subgroups, we note that it suffices to consider finite $p$-groups for either of these conjectures. It is known that class 2 groups satisfy the Amit conjecture  \cite{ainhoa,levy} as do split metacyclic groups   \cite[Prop.~1.3]{levy}. However, besides the examples given by Ashurst, there were no other families of groups known to satisfy the Amit--Ashurst conjecture. In this paper we show that the Amit--Ashurst conjecture holds for finite $p$-groups with a cyclic maximal subgroup. Formally our theorem is:

\begin{prevtheorem}\label{thm:main}
For $p$ a prime, let $G$ be a finite $p$-group with a cyclic maximal subgroup. Then for any word $w$ and $g\in G_w$ we have that $P_{w,G}(g) \geq |G|^{-1}$. 
\end{prevtheorem}

Our proof of Theorem \ref{thm:main} is broken into two parts. First we show that word maps with large images have well-behaved probability distributions on a larger family of finite metacyclic $p$-groups. Then, recalling the classification of finite $p$-groups with a cyclic maximal subgroup, we show that word maps with minimal nontrivial image have appropriately behaved probability distributions. 

\medskip

\noindent
\textit{Notation.}
Throughout, we  use left-normed commutators, for example, $[x,y,z] = [[x,y],z]$. Furthermore we define $[x,y]=xyx^{-1}y^{-1}$.

\subsection*{Acknowledgements} We thank Ainhoa I{\~n}iguez for her initial involvement in the project, and we thank the referee for their helpful comments. This collaboration was initiated at a Functor Categories for Groups meeting, which was supported by a London Mathematical Society Joint Research Group grant. This version of the article has been accepted for publication, after peer review but is not the Version of Record and does not reflect post-acceptance improvements, or any corrections. The Version of Record is available online at: http://dx.doi.org/10.1007/s40879-023-00641-0

\subsection*{Statements and Declarations} The views expressed are those of the authors and do not reflect the official policy or position of the \textsc{ARCYBER}, the Department of the Army, the Department of Defense, or the US Government.

\section{Word Maps over  Metacyclic Groups}

Our basic set up is $G = AB$ is a finite $p$-group, where $A \trianglelefteq G$ with $A$ and $B$
both cyclic. We write $Z$ for the unique subgroup of order~$p$ in~$A$. As
suggested by the notation, the subgroup~$Z$ is central.

As often is the case for nilpotent groups, we would like to prove  the Amit--Ashurst conjecture by taking the  quotient by~$Z$ 
(or any other normal subgroup) and looking at~$G/Z$. Assuming by induction that $P_{w,G/Z}(gZ)$ is at least $|G/Z|^{-1}$, then one would hope that the probability distributes over the coset $gZ$ nicely. This does not work in general, but is an illustrative technique.

First we document a general instance where this idea works.

\begin{lemma}\label{lem:induction} 
Let $G$ be a group and $N$ a normal subgroup such that $G/N$
satisfies the Amit--Ashurst conjecture. Suppose $w$ is a word in $k$ variables and $g \in G_w$ is such that $N_w(gn) = N_w(g)$ for all $n \in N$. Then $N_w(g) \geq |G|^{k-1}$.
\end{lemma}

\begin{proof} Let $w$ be a word in $k$
variables. Write $\overline{G} = G/N$ and we will use the bar notation for elements in~$\overline{G}$. A simple
counting argument, see \cite[page 20]{ashurst}, shows that
$$
N_{w,G/N}(\bar{g}) = \sum_{n \in N} N_{w,G}(gn)\cdot  |N|^{-k}.$$ By
assumption $N_{w, G/N}(\bar{g}) \geq |G/N|^{k-1}$. Thus
$$
N_{w,G}(g)\cdot |N|^{-(k-1)} \geq |G/N|^{k-1}$$ and the result follows.
\end{proof}

\begin{lemma}\label{lem:dichotomy} Let $G$ be a finite $p$-group, such that $G = AB$, where $A\trianglelefteq G$ and both
$A$ and $B$ are cyclic. Let $w$ be a word. If $b^p\in A$ for all $b\in B$, then one of the following holds:
\begin{enumerate}
    \item $G_w = G$,
    \item $G_w\subseteq A$.
\end{enumerate}
\end{lemma}

\begin{proof}
By Hall's collection process \cite[Sec. 11.1]{Hall}, note that $w = x_1^{e_1}\cdots x_\ell^{e_\ell} \kappa$ for some $\ell\in\mathbb{N}\cup\{0\}$, $e_1,\ldots,e_\ell\in\mathbb{Z}$ and word~$\kappa$ that is a product of commutators. Since $G/A$ is cyclic, we see that
$G' \le A$. Hence the image of~$\kappa$ is contained in~$A$. Let $d=\gcd{(e_1,\dots,e_\ell)}$. Since $b^p \in A$, if $p\mid d$ then the image is entirely contained in $A$; otherwise, we have $(p,d)= 1$ and $w$ is a surjective map.
\end{proof}

In particular, if $G_w = G$ then the Amit--Ashurst conjecture follows from \cite[Lem.~7]{CockeHo2}. Therefore, for the rest of this section we consider the case when $G_w \subset G$.

Let $G = AB$ be a finite $p$-group, where $A \trianglelefteq G$ with $A$ and $B$
both cyclic. For ${\bf{b}} = (b_1, \ldots , b_k)\in B^k$, we consider the coset $\mathbf{b}A^k$ in~$G^k$.
Recall that $A$ is normal in~$G$ and as $A=\langle a\rangle$ is cyclic, the cyclic subgroup $B=\langle b\rangle$ acts  
on~$A$ by powers, i.e. $bab^{-1}=a^\ell$ for some $\ell\in\mathbb{Z}$. Hence
over the coset~$\mathbf{b}A^k$, 
a
word~$w$ acts as
\begin{equation}\label{eq:splitting}
w(a_1b_1, \ldots , a_kb_k) = v_{w,{\bf{b}}}(a_1, \ldots , a_k)w({\bf{b}}) = w_{{\bf{b}}}(a_1, \ldots , a_k),
\end{equation}
where $v_{w,{\bf{b}}}$ is a word map over $A$ and $w_{{\bf{b}}}$ is $v_{w,{\bf{b}}}$ times $b^i$ for some~$i$.

\begin{lemma} \label{lem:intersection-in-Z}
Let $G$ be a finite $p$-group, such that $G = AB$, where $A\trianglelefteq G$ and both
$A$ and $B$ are cyclic. Let $Z$ be the unique subgroup of order $p$ in $A$. Suppose
that $A\cap B \le Z$
and that $G/Z$ satisfies the Amit--Ashurst conjecture. Then any word $w$ such that
\[
G_w \not\subseteq Z\quad  \text{and}\quad G_w \subseteq A,
\]
 has $Z \subsetneq G_w$ and satisfies $P_{w,G}(g) \ge |G|^{-1}$ for all $g \in G_w$.
\end{lemma}

\begin{proof}
 For ${\bf{b}} = (b_1, \ldots , b_k)\in B^k$, we have from~\eqref{eq:splitting} that
\[
w(a_1b_1,  \ldots , a_kb_k) = v_{w,\bf{b}}(a_1, \ldots , a_k)w({\bf{b}}).
\]
As $G_w\subseteq A$, we observe that $w(\mathbf{b})\in Z$. Also, as $A$ is abelian $v_{w,\bf{b}} : A^k \rightarrow A$ is a homomorphism
with image, $A_{v_{w,\bf{b}}}$, a subgroup of $A$ and all fibres of $v_{w,\bf{b}}$ are of the same size. Note $A_{v_{w,\bf{b}}}$ is either trivial or contains $Z$, in the second case  $A_{v_{w,\bf{b}}}$ is a union of cosets of $Z$. Thus for any $g \in G_w \backslash Z$ and any fixed ${\bf{b}}$ it follows that 
$N_{ v_{w,\bf{b}}}(g) =  N_{ v_{w,\bf{b}}}(gz) \leq   N_{ v_{w,\bf{b}}}(z)$ for all $z \in Z$, with the inequality coming from the case when
$g \not\in  A_{v_{w,\bf{b}}}w({\bf{b}})$ but $z \in  A_{v_{w,\bf{b}}}w({\bf{b}})$. 
Summing over the $\bf{b}$, gives $N_w(g) = N_w(gz) \leq N_w(z)$. 

We can now proceed by induction.
Let $\overline{G} = G/Z$ (note $Z$ is central in $G$). Suppose 
$g \in G_w \backslash Z$. By hypothesis $N_w(\bar{g}) \geq |\overline{G}|^{k-1}$.
By the previous paragraph $N_w(g) = N_w(gz)$ for all $z \in Z$ and by Lemma~\ref{lem:induction} it follows that 
$N_{w}(g) \geq |G|^{k-1}$. Furthermore, as $N_w(z) \geq N_w(g)$ for all $z \in Z$, both the results follow.
\end{proof}

\section{The Polynomial Argument}
The previous section was largely inspired by Levy~\cite{levy}. The impetus for this section originates in Ashurst's dissertation and together with Lemma \ref{lem:intersection-in-Z} can be used to prove the Amit--Ashurst conjecture for certain metacyclic $p$-groups ~\cite{ashurst}. Recall that Lemma~\ref{lem:intersection-in-Z} completes the idealised inductive proof of the Amit--Ashurst conjecture for finite metacyclic $p$-groups, except for the case when $1\ne G_w\subseteq Z$. To handle this case we consider the following argument.

Let $G$ be a finite nonabelian metacyclic $p$-group. Then from~\cite{King} the group~$G$ can be represented as 
$$G = \langle a, b \mid a^{p^n} = 1, \, b^{p^m} = a^{p^{n- \epsilon}}, \,bab^{-1} = a^r \rangle$$
where $r =p^{n - \delta} + 1$ or $r = p^{n - \delta} -1$ (called positive and negative type respectively) for $n,m\in\mathbb{N}$ and $\epsilon,\delta\in\{0,1,\ldots,n\}$ such that $p^n$ divides both ${p^{n-\epsilon}}(r-1)$ and $r^{p^m}-1$. 
Also, if $b^{p^{n-1}} =1$,  note using $b^{\beta}a^{\alpha} = a^{\alpha r^{\beta}}b^{\beta}$ for $0\le \alpha<p^n$ and $0 \leq \beta < p^m$, we get
\begin{equation}
    \label{eq:useful}
(a^{\alpha}b^{\beta})^{p^{n-1}} = a^{\alpha(1 + r^{\beta} + r^{2\beta} + \cdots + r^{(p^{n-1} - 1)\beta})},
\end{equation}
which will be of later use.

Suppose that a word $w$ maps $G$ to $Z = \langle a^{p^{n-1}} \rangle$, the unique central subgroup of order~$p$ in~$\langle a \rangle$. 
So we have
$$
w (a^{\alpha_1} b^{\beta_1}, \ldots, a^{\alpha_k} b^{\beta_k}) = (a^{p^{n-1}}) ^{f(\boldsymbol{\alpha},\boldsymbol{\beta})}
$$
for some function $f:{\mathbb F}_{p^n}^k \times {\mathbb F}_{p^m}^k \rightarrow {\mathbb F}_p$ where $\boldsymbol{\alpha} = (\alpha_1, \ldots, \alpha_k)$ and $\boldsymbol{\beta}= (\beta_1, \ldots, \beta_k)$ with $0 \leq \alpha_i < p^{n}$ and $0 \leq \beta_i < p^m$.

Suppose for now that we can replace 
$f$ with a polynomial $q: {\mathbb F}_p^k \times {\mathbb F}_p^k \rightarrow {\mathbb F}_p$.

We note the following useful result.

\begin{theorem}[Chevalley-Warning]\cite[Thm.~6.11]{LidlNiederreiter}\label{thm:Chevalley-Warning}
Let $p$ be a prime, let $\ell\in\mathbb{N}$ and let $\widetilde{q}\in\mathbb{F}_p[t_1,\ldots,t_\ell]$ be a polynomial with $d:=\deg(\widetilde{q})<\ell$. If the number $N$ of $(\alpha_1,\ldots,\alpha_\ell)\in\mathbb{F}_p^\ell$ with  $\widetilde{q}(\alpha_1,	\ldots,\alpha_\ell)=0$  satisfies $N\ge 1$, then $N\ge p^{\ell-d}$.
\end{theorem}

For the purposes of the Amit--Ashurst conjecture, we aim to apply the Chevalley-Warning Theorem to some of the polynomials  
$$
\widetilde{q}_i(t_1,\ldots,t_{2k})=q(t_1,\ldots,t_{2k})-i\qquad \text{for each }i\in\mathbb{F}_p;
$$
in this setting, the number of variables is $\ell=2k$.
In order to 
do so, we will  first need that
${\rm deg}(q) <2k$. Either this holds, or if not we can replace $w$ with 
$\widehat{w}(x_1, \ldots, x_{\mu}) = w(x_1, \ldots, x_k)$ where $\mu > k$. We note that, as a word map, the word $\widehat{w}$ sends $(g_1,\ldots,g_{\mu})$ to $w(g_1,\ldots,g_k)$. Hence  $\widehat{w}$ has the same image as $w$ and
also each element in the image has the same probability of occurring since $x_{k+1}, \ldots, x_{\mu}$ can take any value in $G$.

Next, as we will see below, we need a sharper upper bound on ${\rm deg}(q)$.
Write $c$ for the nilpotency class of~$G$. We suppose for now that $\deg(q)\le c$, and as seen in the previous paragraph we may assume that $2k> c$.

 To summarise, we suppose
$$
w (a^{\alpha_1} b^{\beta_1}, \ldots, a^{\alpha_k} b^{\beta_k}) = (a^{p^{n-1}}) ^{q(\overline{\boldsymbol{\alpha}},\overline{\boldsymbol{\beta}})}
$$
for some polynomial $q: {\mathbb F}_p^k \times {\mathbb F}_p^k \rightarrow {\mathbb F}_p$, where $\overline{\boldsymbol{\alpha}} \equiv \boldsymbol{\alpha} \pmod p$ and $\overline{\boldsymbol{\beta}} \equiv \boldsymbol{\beta} \pmod p$, 
 with $\deg(q)\le c<2k$. Suppose $G_w=\{(a^{p^{n-1}})^i\mid i\in I\}\subseteq Z$, for some subset $I\subseteq \mathbb{F}_p$.
Note that for each $i\in I$, the equation 
 $q(\overline{\boldsymbol{\alpha}},\overline{\boldsymbol{\beta}})=i$ has at least one solution. 
 For each $i\in I$, we set $\widetilde{q}_i(\overline{\boldsymbol{\alpha}},\overline{\boldsymbol{\beta}})=q(\overline{\boldsymbol{\alpha}},\overline{\boldsymbol{\beta}})-i$.
 Consequently, by Theorem~\ref{thm:Chevalley-Warning}, for any $g\in G_w
 $, the set of solutions for $w=g$ (as seen as a word in~$k$ variables) is of size at least $p^{k(n-1)}\cdot p^{k(m-1)}\cdot p^{2k-c}$. Indeed, recalling that
 $$G = \langle a, b \mid a^{p^n} = 1, \, b^{p^m} = a^{p^{n- \epsilon}}, \,bab^{-1} = a^r \rangle,$$
 for each solution $(\overline{\boldsymbol{\alpha}},\overline{\boldsymbol{\beta}})$ of~$\widetilde{q}_i$, where $i\in I$ is such that $g=(a^{p^{n-1}})^i$, we have $p^{k(n-1)}$ choices for $\boldsymbol{\alpha}$ with $\boldsymbol{\alpha} \equiv \overline{\boldsymbol{\alpha}} \pmod p$ and $p^{k(m-1)}$ choices for $\boldsymbol{\beta}$ with $\boldsymbol{\beta} \equiv \overline{\boldsymbol{\beta}} \pmod p$. Hence, 
\[
P_{w,G}(g)\ge \frac{p^{k(n-1)}\cdot p^{k(m-1)}\cdot p^{ 2k-c}}{p^{k(n-1)}\cdot p^{k(m-1)}\cdot p^{2k}}=\frac{1}{p^{c}}>|G|^{-1},
\]
which gives us the Amit--Ashurst bound for the case $G_w\subseteq Z$.

Therefore it remains to determine the instances when we can  replace 
$f(\boldsymbol{\alpha},\boldsymbol{\beta})$ with a polynomial $q: {\mathbb F}_p^k \times {\mathbb F}_p^k \rightarrow {\mathbb F}_p$ of degree at most~$c$.
This will be done in the next section, but we record here some key results for our goal.

Recall from Hall's collecting process \cite[Sec. 11.1]{Hall} 
that
$$w(x_1, \ldots, x_k) = x_1^{e_1} \cdots x_k^{e_k} \kappa(x_1,\ldots,x_k)$$
where $ \kappa(x_1,\ldots,x_k) \in \textbf{F}_k'$ and $e_1,\ldots,e_k\in\mathbb{Z}$. As $w$ maps to $Z$ we have $p^{n-1}$ divides~$e_i$, for $1\le i\le k$. Next, since $G$ is metabelian 
we only need
to consider the basic commutators $\gamma_{\ell}(g_1, \ldots, g_{\ell}) = [g_1, g_2, \ldots, g_{\ell}]$ for $\ell\ge 2$.
Thus we now analyse the commutator structure of~$G$, 
remembering that we define a commutator
$[g,h]$ as $ghg^{-1}h^{-1}$. As any element in~$G$ can be written as $a^{\alpha_i} b^{\beta_j}$ for some
$0 \leq \alpha_i < p^n$ and $0 \leq \beta_j < p^m$, we introduce the following notation: for $\ell\ge 2$,
$$
\gamma_{\ell}(\boldsymbol{\alpha}, \boldsymbol{\beta}) = [a^{\alpha_1}b^{\beta_1}, a^{\alpha_2} b^{\beta_2}, \ldots, a^{\alpha_{\ell}} b^{\beta_{\ell}}]
= a^{q_{\ell}(\boldsymbol{\alpha}, \boldsymbol{\beta})}
$$
for $\boldsymbol{\alpha} = (\alpha_1, \ldots, \alpha_{\ell})$ and $\boldsymbol{\beta}= (\beta_1, \ldots, \beta_{\ell})$, where $q_\ell(\boldsymbol{\alpha},\boldsymbol{\beta})$ is some function in $\boldsymbol{\alpha}$ and $\boldsymbol{\beta}$. 
Note our abuse of notation where if $\boldsymbol{\alpha} = (\alpha_1, \ldots, \alpha_t)$ and $\boldsymbol{\beta} = (\beta_1, \ldots, \beta_t)$ and $\ell <t$ then
$q_\ell(\boldsymbol{\alpha}, \boldsymbol{\beta}) = q_\ell((\alpha_1, \ldots, \alpha_{\ell}), (\beta_1. \ldots, \beta_{\ell}))$.

\begin{lemma} \label{lem:metacyclic} In the set-up above, suppose $G$ is a metacyclic $p$-group with
$$
G = \langle a, b\mid a^{p^n} = 1,\, b^{p^m} = a^{p^{n- \epsilon}},\, bab^{-1} = a^r \rangle.
$$
Then 
$$
q_2(\boldsymbol{\alpha} , \boldsymbol{\beta}) = \alpha_1(1-r^{\beta_2}) - \alpha_2(1-r^{\beta_1})
$$
and
$$
q_{\ell}(\boldsymbol{\alpha}, \boldsymbol{\beta}) = q_2(\boldsymbol{\alpha}, \boldsymbol{\beta})(1-r^{\beta_3}) \cdots (1-r^{\beta_{\ell}})$$
for $\ell \geq 3$.
\end{lemma}

\begin{proof}
As $ba = a^rb$ it follows that $b^{\beta} a^{\alpha} = a^{\alpha r^{\beta}}b^{\beta}$. Thus
\begin{eqnarray*}
[a^{\alpha_1}b^{\beta_1}, a^{\alpha_2}b^{\beta_2}] & = & a^{\alpha_1}b^{\beta_1} a^{\alpha_2}b^{\beta_2} b^{-\beta_1}a^{-\alpha_1}{b^{-\beta_2}
a^{-\alpha_2}} \\
& = & a^{\alpha_1} a^{\alpha_2r^{\beta_1}}a^{-\alpha_1r^{\beta_2}}a^{-\alpha_2} \\
& = & a^{\alpha_1(1 - r^{\beta_2}) - \alpha_2(1 - r^{\beta_1})}.
\end{eqnarray*} 
This gives the expression for $q_2(\boldsymbol{\alpha}, \boldsymbol{\beta})$. We then proceed by induction. Suppose $\ell \geq 3$. We have
\begin{align*}
\gamma_\ell(\boldsymbol{\alpha}, \boldsymbol{\beta}) & =  [a^{q_{\ell-1}(\boldsymbol{\alpha}, \boldsymbol{\beta})}, a^{\alpha_\ell} b^{\beta_\ell} ] \\       & =  a^{q_{\ell-1}(\boldsymbol{\alpha}, \boldsymbol{\beta})(1 - r^{\beta_\ell})}.\qedhere
\end{align*}
\end{proof}

We can apply the lemma in various different cases. 

\begin{lemma}\label{Lem:dihedral-quaternion} 
Suppose $G$ is a metacyclic 2-group of one of the following types:
$$
G = \langle a,b \mid a^{2^n} = 1,\, b^{2^m} = a^{2^{n-1}},\, bab^{-1}= a^{-1} \rangle 
$$
or 
$$
G = \langle a,b \mid a^{2^n} = b^{2^m} =1,\, bab^{-1}  = a^{-1}\rangle.
$$
Then for $\ell\ge 2$,
$$
q_{\ell}(\boldsymbol{\alpha} , \boldsymbol{\beta}) = 2^{\ell-1}  \bar{\beta}_{\ell} \cdots \bar{\beta}_3(\alpha_1 \bar{\beta}_2 - \alpha_2 \bar{\beta}_1)
$$
where $\bar{\beta}_i \equiv \beta_i \pmod 2$ and $\bar{\beta}_i\in \mathbb{F}_2$ for $1\le i\le\ell$.
\end{lemma} 

\begin{proof}
We apply Lemma~\ref{lem:metacyclic} with $r = -1$ (equivalently $r = 2^n -1$).  So
$$
q_2(\boldsymbol{\alpha}, \boldsymbol{\beta}) = \alpha_1 (1 - (-1)^{\beta_2}) - \alpha_2(1 - (-1)^{\beta_1}) = 
2(\alpha_1 \bar{\beta}_2 - \alpha_2 \bar{\beta}_1),
$$
and 
\[
q_{\ell}(\boldsymbol{\alpha}, \boldsymbol{\beta}) = q_{\ell-1}(\boldsymbol{\alpha}, \boldsymbol{\beta})(1 - (-1)^{\beta_{\ell}}) = 2\bar{\beta}_{\ell}\cdot q_{\ell-1}(\boldsymbol{\alpha}, \boldsymbol{\beta}). \qedhere
\]
\end{proof}

In particular the above lemma can be applied to dihedral groups  (when $b^2 =1$) and generalised
quaternion groups (when $b^2 = a^{2^{n-1}}$);   in these cases
$\bar{\beta_i} = \beta_i$. 

The following lemma considers semidihedral groups.

\begin{lemma} \label{lem:semidihedral}
Let $G = \langle a,b\mid a^{2^{n}} = b^2 =1,\, bab^{-1} = a^{2^{n-1} -1} \rangle$ for $n\ge 2$. Then for $\ell\ge 2$,
$$
q_{\ell}(\boldsymbol{\alpha}, \boldsymbol{\beta}) = 2^{\ell-1} (1 - 2^{n-2})^{\ell-1} \beta_3 \cdots \beta_{\ell}(\alpha_1 \beta_2 - \alpha_2 \beta_1).
$$
\end{lemma}

\begin{proof} Here we apply Lemma~\ref{lem:metacyclic} with $r = 2^{n-1} -1$, and we obtain
\begin{eqnarray*}
q_2 (\boldsymbol{\alpha}, \boldsymbol{\beta}) & = & \alpha_1(1 - (2^{n-1} - 1)^{\beta_2}) - \alpha_2(1-(2^{n-1} -1)^{\beta_1}) \\ 
 & = & (2 - 2^{n-1})(\alpha_1\beta_2 - \alpha_2 \beta_1)
\end{eqnarray*}
and for $\ell > 2$ 
\begin{align*}
q_{\ell}(\boldsymbol{\alpha}, \boldsymbol{\beta}) & =  q_2(\boldsymbol{\alpha}, \boldsymbol{\beta}) (1- (2^{n-1} - 1)^{\beta_3}) \cdots (1- (2^{n-1} -1)^{\beta_{\ell}})\\
                    & =  (2 - 2^{n-1})^{\ell-1}\beta_3 \cdots \beta_{\ell}(\alpha_1 \beta_2 - \alpha_2 \beta_1). \qedhere
\end{align*}
\end{proof}

Note that in the above result, we have $\beta_\ell\in\mathbb{F}_2$ since $b$ has order~2.

\begin{lemma} \label{lem:class-2} Suppose $G$ is a metacyclic $p$-group with
$$
G = \langle a, b\mid a^{p^n} = 1, \,b^{p^m} = a^{p^{n- \epsilon}}, \,bab^{-1} = a^{1 + p^{n-1}} \rangle.
$$
Then either $\epsilon=0$ 
or $G\cong Q_8$. 

Furthermore,
$$
q_2(\boldsymbol{\alpha}, \boldsymbol{\beta}) = p^{n-1}(\alpha_2 \beta_1 - \alpha_1 \beta_2).
$$
\end{lemma}

\begin{proof} For the first claim, we observe that $G$ is of class 2. The claim now follows from \cite[Thm.~1.1]{beuerle}. For the second claim, we proceed as above, noting  that $a^{p^{n}}=1$.
\end{proof}

\begin{remark}\label{rmk:split}
For a group $G$ as in the above lemma, when $p$ is odd and if $m=1$ (equivalently, if $G$ is a nonabelian $p$-group of order~$p^{n+1}$ with an element $a$ of order~$p^{n}$, for an odd prime $p$), then it follows from \cite[Thm.~7.18]{Roman}
that $G$ is a semidirect product of $\langle a \rangle$ and $\langle b \rangle$ where $o(b) = p$ and $a^b = a^{1+p^{n-1}}$. 
\end{remark}


\section{The Amit--Ashurst Conjecture for finite $p$-groups with cyclic maximal subgroups}

A nonabelian finite $p$-group $G$ has a cyclic maximal subgroup if and only if there are cyclic subgroups $A$ and $B$ with $A\trianglelefteq G$, the product $AB = G$, and $|G:A| = p$. There are 5 such families of $p$-groups 
\cite[Thm.~5.4.4]{Gorenstein}. 

\subsection*{The nonabelian 2-groups with a cyclic maximal subgroup are:}
\begin{enumerate}
    \item the generalised quaternion groups
    $$
    Q_{2^{n+1}} = \langle a, b \mid a^{2^{n}}=1,\, a^{2^{n-1}} = b^2,\, a^b = a^{-1} \rangle,\qquad\text{for $n\ge 2$,}
    $$
    \item the dihedral groups
    $$
   \, D_{2^{n+1}}= \langle a,b \mid a^{2^{n}}= b^2=1, \,a^b = a^{-1}\rangle,\qquad\,\,\,\,\,\qquad\text{for $n\ge 2$,}
    $$
    \item the semidihedral groups
    $$
    SD_{2^{n+1}} = \langle a, b\mid a^{2^{n}}= b^2=1,\, a^{b} = a^{2^{n-1}-1}\rangle,\quad\qquad\text{for $n\ge 3$,}
    $$
    \item the groups of the form
    $$
   \,\qquad G = \langle a, b \mid a^{2^{n}}= b^2=1,\, a^b = a^{2^{n-1}+1}\rangle, \quad\qquad\text{for $n\ge 2$}.
    $$
\end{enumerate}
\subsection*{For an odd prime $p$, the nonabelian $p$-groups with a cyclic maximal subgroup are:}
\begin{enumerate}
    \item[(5)] the groups of the form 
    $$
   \qquad G= \langle a, b \mid a^{p^{n}}= b^{p}=1,\, a^{b} = a^{p^{n-1}+1}\rangle, \!\quad\qquad\text{for $n\ge 2$}.
    $$
\end{enumerate}

In addition to differences between $2$ and odd primes, nonabelian finite $p$-groups with cyclic maximal subgroups can be broadly separated into two categories: those of class 2 and those of maximal class. The dihedral, generalised quaternions, and semidihedral groups have maximal class, while the other nonabelian finite $p$-groups with cyclic maximal subgroups have class~2. We observe that quotienting by the central subgroup of order~$p$ results in either a maximal class subgroup of smaller rank, or an abelian group. This allows us to use Lemma \ref{lem:intersection-in-Z}, assuming that the Amit--Ashurst conjecture holds for the quotient. We note that the Amit--Ashurst conjecture holds for all finite abelian groups.

\begin{proposition}\label{pr:polynomial}
For $p$ a prime, let $G$ be a finite nonabelian $p$-group with a cyclic maximal subgroup. Then, for any word $w$ such that
$
G_w\subseteq Z$, where $Z$ is the minimal central subgroup of order~$p$, we have
 $P_{w,G}(g) \ge |G|^{-1}$ for all $g \in G_w$.
\end{proposition}

\begin{proof}
Suppose $w$ is a word in $k$ variables such that $1\ne G_w\subseteq Z$. Then, using the argument and the notation from the previous section, it suffices to show that
$$
w (a^{\alpha_1} b^{\beta_1}, \ldots, a^{\alpha_k} b^{\beta_k}) = \big(a^{\frac{o(a)}{p}}\big) ^{q(\overline{\boldsymbol{\alpha}},\overline{\boldsymbol{\beta}})}
$$
for some polynomial $q: {\mathbb F}_p^k \times {\mathbb F}_p^k \rightarrow {\mathbb F}_p$ of degree at most the nilpotency class of~$G$, where $\overline{\boldsymbol{\alpha}} \equiv \boldsymbol{\alpha} \pmod p$ and $\overline{\boldsymbol{\beta}} \equiv \boldsymbol{\beta} \pmod p$.

Note from the above classification of finite nonabelian $p$-groups with a cyclic maximal subgroup, we have that either
(1) $G$ is dihedral, (2) $G$ is generalised quaternion, (3) $G$ is semidihedral or 
(4)-and-(5) $G = \langle a,b\mid a^{p^{n}} = b^p= 1, \,bab^{-1} = a^{1 + p^{n-1}}\rangle$.
In all cases $|G| = p^{n+1}$ and we are assuming our word $w$ maps to $\langle a^{p^{n-1}} \rangle$. Recall
$$
w(x_1, \ldots, x_k) = x_1^{e_1} \cdots x_k^{e_k} \cdot \kappa(x_1, \ldots, x_k).
$$
It follows that $p^{n-1}$ divides $e_i$. 
Also using \eqref{eq:useful}, one can check that
\begin{align*}
    &(a^{\alpha_i} b^{\beta_i})^{p^{n-1}}
     &=\begin{cases}
     (a^{\alpha_i})^{2^{n-1}(1-\overline{\beta}_i)} & \text{if $G$ is of type (1), (2) or (3)}\\
    \big((a^{\alpha_i})^{p^{n-1}}\big)^{1+\overline{\beta}_i +2\overline{\beta}_i+\cdots +(p^{n-1}-1)\overline{\beta}_i}& \text{if $G$ is of type (4) or (5)}
    \end{cases}
\end{align*}
 where $\beta_i\equiv \overline{\beta}_i \pmod p$ and $\overline{\beta}_i\in \mathbb{F}_p$. Hence
$$
w(a^{\alpha_1} b^{\beta_1}, \ldots, a^{\alpha_k} b^{\beta_k}) 
= (a^{p^{n-1}})^{\rho(\overline{\boldsymbol{\alpha}}, \overline{\boldsymbol{\beta}}  )} \cdot \kappa(a^{\alpha_1} b^{\beta_1}, \ldots, a^{\alpha_k} b^{\beta_k}),
$$
for some polynomial 
$\rho: {\mathbb F}_p^k \times {\mathbb F}_p^k \rightarrow {\mathbb F}_p$ of degree at most~2. 
 Recall also that the word~$\kappa$ is a product of basic commutators of length at most the nilpotency class of~$G$. Note $\kappa$ is of this form since $G$ is metabelian.
We consider
$$
\kappa(a^{\alpha_1} b^{\beta_1}, \ldots, a^{\alpha_k} b^{\beta_k})  = \prod_{j=2}^{n} \Gamma_j(\boldsymbol{\alpha}, \boldsymbol{\beta}),
$$
where $\Gamma_j(\boldsymbol{\alpha}, \boldsymbol{\beta})$ is the product of all basic commutators in~$j$ variables in the word~$\kappa$.
 Also, by abuse of  notation, we do not distinguish the fact that the effective set of variables for each of the maps $\rho$, $\kappa$ and all $\Gamma_j$ is some subset of the respective given set of variables.

If $G$ is of type (4) or (5),
only $\Gamma_2(\boldsymbol{\alpha}, \boldsymbol{\beta})$ appears and the result follows since $q_2(\boldsymbol{\alpha}, \boldsymbol{\beta} )$ is of the required form by Lemma~\ref{lem:class-2}. 
In the remaining three cases, we have that $G$ is a 2-group. As
$G_w\subseteq\langle a^{2^{n-1}} \rangle$, we deduce that $ \Gamma_2(\boldsymbol{\alpha}, \boldsymbol{\beta})\in\langle a^{2^{n-1}}\rangle$ for all choices of $\boldsymbol{\alpha}, \boldsymbol{\beta}$.
 Indeed, for $\Gamma_2(\boldsymbol{\alpha}, \boldsymbol{\beta})$ to be nontrivial, 
we need some of $\alpha_{i_1}, \alpha_{i_2},\beta_{i_1},\beta_{i_2}$ to be nonzero, for some choice of distinct $i_1,i_2\in\{1,\ldots, k\}$; compare Lemmas~\ref{Lem:dihedral-quaternion} and~\ref{lem:semidihedral}, noting also that we may assume that  $\beta_{i_1},\beta_{i_2}\in\{0,1\}$. For  $1\le i_1<i_2\le k$, let
\[
\Lambda_2(i_1,i_2)=\Gamma_2\big((0,\ldots,0,\alpha_{i_1},0,\ldots,0,\alpha_{i_2},0,\ldots,0),(0,\ldots,0,\beta_{i_1},0,\ldots,0,\beta_{i_2},0,\ldots,0)\big)
\]
and we set
$$\mathbf{I}=\{(i_1,i_2)\mid \Lambda_2(i_1,i_2)\ne 1 \text{ for some }\alpha_{i_1}, \alpha_{i_2},\beta_{i_1},\beta_{i_2}\}.$$
Then 
\[
\Gamma_2(\boldsymbol{\alpha}, \boldsymbol{\beta})=\prod_{(i_1,i_2)\in\mathbf{I}}\Lambda_2(i_1,i_2).
\]For each fixed choice of such $i_1,i_2$ with $\Gamma_2(\boldsymbol{\alpha}, \boldsymbol{\beta})\ne 1$, set $\alpha_i=\beta_i=0$ for all $i\in\{1,\ldots, k\}\backslash\{i_1,i_2\}$. It then follows that 
 $ \Gamma_j(\boldsymbol{\alpha}, \boldsymbol{\beta})=1$ for all $j>2$. Next, suppose that $\Gamma_2(\boldsymbol{\alpha}, \boldsymbol{\beta})$ is a product of commutators of lengths $\ell_1,\ldots,\ell_d$ for some $d\in\mathbb{N}$ with $2\le \ell_1<\cdots<\ell_d\le n$; here $n$ is also the nilpotency class of~$G$.  For $2\le \ell\le n$,  it follows from Lemmas~\ref{Lem:dihedral-quaternion} and~\ref{lem:semidihedral}  that $2^{\ell-1}\mid q_\ell(\boldsymbol{\alpha}, \boldsymbol{\beta})$ for all $\boldsymbol{\alpha}, \boldsymbol{\beta}$. 
 Hence if $ \Lambda_2(i_1,i_2)$ is nontrivial for nonzero $\alpha_{i_1},\beta_{i_2}$ with $\beta_{i_1}$ being zero, we have that
 \[
 \Lambda_2(i_1,i_2)=a^{\alpha_{i_1}
 \big(2^{\ell_1-1}\mu_1+\cdots+ 2^{\ell_d-1}\mu_d\big)}
 \]
 for some $\mu_1,\ldots,\mu_d\in\mathbb{Z}$; here the $\mu_t$ correspond to the  exponent of the commutator $[a^{\alpha_{i_2}}b^{\beta_{i_2}},a^{\alpha_{i_1}}b^{\beta_{i_1}}, a^{\alpha_{i_2}}b^{\beta_{i_2}},\overset{\ell_t-2}\ldots,a^{\alpha_{i_2}}b^{\beta_{i_2}} ]$ in $\Lambda_2(i_1,i_2)$, equivalently in $\Gamma_2(\boldsymbol{\alpha}, \boldsymbol{\beta})$, for $t\in\{1,\ldots, d\}$. The condition $G_w\subseteq\langle a^{2^{n-1}} \rangle$ now implies that
\[
2^{\ell_1-1}\mu_1+\cdots+ 2^{\ell_d-1}\mu_d\equiv 0\pmod {2^{n-1}};
\]
indeed,  let for example $\alpha_{i_1}=1$. If $ \Lambda_2(i_1,i_2)$ is nontrivial for nonzero $\alpha_{i_2},\beta_{i_1}$ with  $\beta_{i_2}$ being zero, then
 \[
  \Lambda_2(i_1,i_2)=a^{-\alpha_{i_2}
 \big(2^{\ell_1-1}\nu_1+\cdots+ 2^{\ell_d-1}\nu_d\big)}
 \]
 for some $\nu_1,\ldots,\nu_d\in\mathbb{Z}$, here the $\nu_t$ correspond to the  exponent of the commutator $[a^{\alpha_{i_1}}b^{\beta_{i_1}},a^{\alpha_{i_2}}b^{\beta_{i_2}}, a^{\alpha_{i_1}}b^{\beta_{i_1}},\overset{\ell_t-2}\ldots,a^{\alpha_{i_1}}b^{\beta_{i_1}} ]$ in $\Lambda_2(i_1,i_2)$. Similarly
 \[
2^{\ell_1-1}\nu_1+\cdots+ 2^{\ell_d-1}\nu_d\equiv 0\pmod {2^{n-1}}.
\]
Finally, if $ \Lambda_2(i_1,i_2)$ is nontrivial for nonzero $\beta_{i_1},\beta_{i_2}$, the condition $G_w\subseteq\langle a^{2^{n-1}} \rangle$, together with the fact that
 \[
 \Lambda_2(i_1,i_2)=a^{(\alpha_{i_1}-\alpha_{i_2})
 \big(2^{\ell_1-1}\lambda_1+\cdots+ 2^{\ell_d-1}\lambda_d\big)}
 \]
 for some $\lambda_1,\ldots,\lambda_d\in\mathbb{Z}$, yields that 
  \[
2^{\ell_1-1}\lambda_1+\cdots+ 2^{\ell_d-1}\lambda_d\equiv 0\pmod {2^{n-1}};
\]
here the $\lambda_t$ correspond to the total exponent of all  commutators of length~$\ell_t$
in $\Lambda_2(i_1,i_2)$.
Therefore  $ \Gamma_2(\boldsymbol{\alpha}, \boldsymbol{\beta})= (a^{2^{n-1}})^{\overline{q}_2(\overline{\boldsymbol{\alpha}}, \overline{\boldsymbol{\beta}} )}$ for some polynomial $\overline{q}_2:{\mathbb F}_p^k \times {\mathbb F}_p^k \rightarrow {\mathbb F}_p$ of degree at most~$3$. Specifically, we have
\begin{align*}
2^{n-1}\cdot\overline{q}_2(\overline{\boldsymbol{\alpha}}, \overline{\boldsymbol{\beta}} )&=2^{n-1}\cdot\overline{q}_2(\overline{\boldsymbol{\alpha}}, \boldsymbol{\beta} )\\
&=\sum_{(i_1,i_2)\in\mathbf{I}}\overline{\alpha}_{i_1}\beta_{i_2}(1-\beta_{i_1})
 \big(2^{\ell_1-1}\mu_1+\cdots+ 2^{\ell_d-1}\mu_d\big)\\
&\quad -\sum_{(i_1,i_2)\in\mathbf{I}}\overline{\alpha}_{i_2}\beta_{i_1}(1-\beta_{i_2})
 \big(2^{\ell_1-1}\nu_1+\cdots+ 2^{\ell_d-1}\nu_d\big)\\
 &\quad +\sum_{(i_1,i_2)\in\mathbf{I}}(\overline{\alpha}_{i_1}-\overline{\alpha}_{i_2})\beta_{i_1}\beta_{i_2}
 \big(2^{\ell_1-1}\lambda_1+\cdots+ 2^{\ell_d-1}\lambda_d\big),
\end{align*}
where the $\mu_t,\nu_t,\lambda_t$ depend on $(i_1,i_2)\in\mathbf{I}$.
Repeating this argument recursively for increasing values of $j\ge 2$, we deduce that
$ \Gamma_j(\boldsymbol{\alpha}, \boldsymbol{\beta})= (a^{2^{n-1}})^{\overline{q}_j(\overline{\boldsymbol{\alpha}}, \overline{\boldsymbol{\beta}} )}$ for some polynomial $\overline{q}_j:{\mathbb F}_p^k \times {\mathbb F}_p^k \rightarrow {\mathbb F}_p$ of degree at most~$j+1$, for all $2\le j< n$. 
Lastly, from Lemmas~\ref{Lem:dihedral-quaternion} and~\ref{lem:semidihedral}, it is clear that $ \Gamma_n(\boldsymbol{\alpha}, \boldsymbol{\beta})= (a^{2^{n-1}})^{\overline{q}_n(\overline{\boldsymbol{\alpha}}, \overline{\boldsymbol{\beta}} )}$ with $\overline{q}_n:{\mathbb F}_p^k \times {\mathbb F}_p^k \rightarrow {\mathbb F}_p$ of degree at most~$n$. 
\end{proof}

\begin{remark}
It is worth pointing out that the condition $1\ne G_w\subseteq Z$ in Proposition~\ref{pr:polynomial} can be replaced with $G_w=Z$. Indeed, if $p=2$ then this is clear since $|Z|=2$. If $p$ is odd, the group~$G=AB$ is of type (5) and hence here $A\cap B=1$; compare Remark~\ref{rmk:split}. Therefore, by the same argument (and using the same notation as) in the proof of Lemma~\ref{lem:intersection-in-Z}, it follows that $w(\mathbf{b})=1$ and hence $Z\le A_{v_{w,\mathbf{b}}}$.
\end{remark}

\begin{theorem}\label{thm:class-2}
For $p$ a prime, let $G$ be a  finite $p$-group  of nilpotency class 2 with a cyclic maximal subgroup. 
 Then $G$ satisfies the Amit--Ashurst conjecture.
\end{theorem}

\begin{proof}
We can assume that $G$ is nonabelian and hence $|G| \ge 8$.

By the above classification, we have that $G$ is either the quaternion group~$Q_8$ or of type (4) or (5). 
Since 
$\langle a^{p^{n-1}}\rangle\le Z(G)$ coincides with the  cyclic subgroup~$Z$ of~$A$ of order~$p$, the quotient $G/Z\cong C_{p^{n-1}}\times C_p$ is abelian. 
Now Lemma~\ref{lem:dichotomy}, Lemma~\ref{lem:intersection-in-Z} and  Proposition~\ref{pr:polynomial} show that the Amit--Ashurst conjecture holds for~$G$.
\end{proof}

\begin{theorem}\label{thm:max-class}
 Let $G$ be a maximal class $2$-group. 
 Then $G$ satisfies the Amit--Ashurst conjecture.
\end{theorem}

\begin{proof}
A maximal class $2$-group is contained in one of the following families: the dihedral groups; the generalised quaternion groups; or the semidihedral groups.  Likewise, noting that the quotient $G/Z$ is either of class~2 or isomorphic to a dihedral group, the result follows from  Lemma~\ref{lem:dichotomy}, Lemma~\ref{lem:intersection-in-Z}  and  Proposition~\ref{pr:polynomial}.
\end{proof}

Combining Theorems~\ref{thm:class-2} and \ref{thm:max-class} proves Theorem~\ref{thm:main}.

\end{document}